\documentclass[11pt]{article}

\usepackage{tikz}
\usepackage{nameref}
\usepackage[shortlabels,inline]{enumitem}
\usepackage{empheq}
\usepackage[shortlabels]{enumitem}
\setlist[enumerate]{nosep}
\usepackage[doc]{optional}
\usepackage{xcolor}
\usepackage[colorlinks=true,
            linkcolor=refkey,
            urlcolor=lblue,
            citecolor=red]{hyperref}
\usepackage{float}
\usepackage{soul}
\usepackage{graphicx}
\definecolor{labelkey}{rgb}{0,0.08,0.45}
\definecolor{refkey}{rgb}{0,0.6,0.0}
\definecolor{Brown}{rgb}{0.45,0.0,0.05}
\definecolor{lime}{rgb}{0.00,0.8,0.0}
\definecolor{lblue}{rgb}{0.5,0.5,0.99}

 \usepackage{mathpazo}

\colorlet{hlcyan}{cyan!30}

\usepackage{stmaryrd}
\usepackage{amssymb}
\makeatletter
\def\namedlabel#1#2{\begingroup
   \def\@currentlabel{#2}%
   \label{#1}\endgroup
}
\makeatother

\newcommand{\seppthree}{\setlength{\itemsep}{-3pt}}

\newcommand{\nnn}{\ensuremath{{n\in{\mathbb N}}}}

\newcommand{\fenv}[1]%
{\ensuremath{\,\overrightarrow{\operatorname{env}}_{#1}}}
\newcommand{\benv}[1]%
{\ensuremath{\,\overleftarrow{\operatorname{env}}_{#1}}}

\newcommand{\scal}[2]{\left\langle{#1},{#2}  \right\rangle}

\newcommand{\RR}{\ensuremath{\mathbb R}}
\newcommand{\RX}{\ensuremath{\,\left]-\infty,+\infty\right]}}

\newcommand{\NN}{\ensuremath{\mathbb N}}

\newcommand{\ran}{\ensuremath{{\operatorname{ran}}\,}}

\newcommand{\cconv}{\ensuremath{\overline{\operatorname{conv}}\,}}
\newcommand{\cspan}{\ensuremath{\overline{\operatorname{span}}\,}}

\newcommand{\Id}{\ensuremath{\operatorname{Id}}}
\newcommand{\whs}{\text{(whs)~}}

{\begin{list}{}{%
\settowidth{\labelwidth}{\textrm{#1~}}%
\setlength{\leftmargin}{\labelwidth+\labelsep}}}%requires macro calc.sty
{\end{list}}
\usepackage{amsthm}
\usepackage[capitalize,nameinlink]{cleveref}
\crefname{equation}{}{equations}
\crefname{chapter}{Appendix}{chapters}
\crefname{item}{}{items}
\crefname{enumi}{}{}
\newtheorem{theorem}{Theorem}[section]
\newtheorem{lemma}[theorem]{Lemma}

\newtheorem{corollary}[theorem]{Corollary}

\newtheorem{proposition}[theorem]{Proposition}

\newtheorem{definition}[theorem]{Definition}

\newtheorem{example}[theorem]{Example}
\newtheorem{ex}[theorem]{Example}

\newtheorem{remark}[theorem]{Remark}

\providecommand{\abs}[1]{\lvert#1\rvert}
\providecommand{\norm}[1]{\lVert#1\rVert}

\providecommand{\RR}{\mathbb{R}}

\providecommand{\ran}{\operatorname{ran}}

\providecommand{\Sign}{\operatorname{Sign}}
\providecommand{\sign}{\operatorname{sign}}

\providecommand{\Id}{\operatorname{{ Id}}}

\providecommand{\NN}{\mathbb{N}}
\providecommand{\BB}[2]{\operatorname{ball}(#1;#2)}

\providecommand{\ran}{\operatorname{ran}}

\providecommand{\Id}{\operatorname{Id}}

\providecommand{\J}{{\,\operatorname{J}}}
\newcommand{\cran}{\ensuremath{\overline{\operatorname{ran}}\,}}

\providecommand{\RR}{\mathbb{R}}
\providecommand{\NN}{\mathbb{N}}

\definecolor{myblue}{rgb}{.8, .8, 1}
  \newcommand*\mybluebox[1]{%
    \colorbox{myblue}{\hspace{1em}#1\hspace{1em}}}

\allowdisplaybreaks % or locally if problems {\allowdisplaybreaks
\begin{document}
%-------------------------------------------------------------------------

%\tikzstyle{decision} = [diamond, draw, fill=blue!50]
%\tikzstyle{line} = [draw, -stealth, thick]
%\tikzstyle{elli}=[draw, ellipse, fill=red!50,minimum height=8mm, text width=5em, text centered]
%\tikzstyle{block} = [draw, rectangle, fill=blue!50, text width=8em, text centered, minimum height=15mm, node distance=10em]
%

\author{
Heinz H.\ Bauschke\thanks{
Mathematics, University
of British Columbia,
Kelowna, B.C.\ V1V~1V7, Canada. E-mail:
\texttt{heinz.bauschke@ubc.ca}.},~ 
Walaa M.\ Moursi\thanks{
  Department of Electrical Engineering,
  Stanford University,
  350 Serra Mall, Stanford, CA 94305,
  USA.
  %and
  %Mansoura University, Faculty of Science,
  %Mathematics Department,
  %Mansoura 35516, Egypt.
  E-mail: \texttt{wmoursi@stanford.edu}.}
~and~
Xianfu Wang\thanks{
Mathematics, University
of British Columbia,
Kelowna, B.C.\ V1V~1V7, Canada. E-mail:
\texttt{shawn.wang@ubc.ca}.}
}

\title{\textsf{
Maximally monotone operators with ranges\\
whose closures are not convex and an
answer\\ to a recent question by Stephen Simons 
}
}

\date{September 6, 2019}

\maketitle

\begin{abstract}
In his recent \emph{Proceedings of the AMS} paper
``Gossez's skew linear map and its pathological 
maximally monotone multifunctions'',
Stephen Simons proved that the closure of the
range of the sum of the Gossez operator and a multiple
of the duality map is not convex whenever the scalar
is between $0$ and $4$. The problem of the convexity of
that range when the scalar is equal to $4$ was explicitly
stated. In this paper, we answer this question in the negative
for any scalar greater than or equal to $4$.
We derive this result from an abstract framework that allows us
to also obtain a corresponding result for the Fitzpatrick-Phelps
integral operator. 
\end{abstract}
{ 
%\small
\noindent
{\bfseries 2010 Mathematics Subject Classification:}
{Primary 
47H05, 
Secondary 
46B20, 
47A05 
}

\noindent {\bfseries Keywords:}
Duality map, 
Fitzpatrick-Phelps operator,
Gossez operator, 
maximally monotone operator,
range,
skew operator
}

\section{Introduction}

Throughout, we assume that 
\begin{empheq}[box=\mybluebox]{equation}
\text{$X$ is
a real Banach space with dual pairing
$\scal{\cdot}{\cdot}\colon X\times X^*\to\RR$, }
\end{empheq}
where $X^*$ is the dual space of $X$. 
The 
\emph{duality mapping} $\J$ of $X$ is the subdifferential operator of
the function $\tfrac{1}{2}\|\cdot\|^2\colon X\to\RR$; it satisfies
\begin{equation}
(\forall (x,x^*)\in X\times X^*)\quad 
x^*\in \J x \; \Leftrightarrow \; \scal{x}{x^*}=\|x\|^2 = \|x^*\|^2.
\label{e:J}
\end{equation}
Now let $A\colon X\to X^*$ be a bounded linear monotone operator, i.e.,
$(\forall x\in X)(\forall y\in X)$ $\scal{x-y}{Ax-Ay} \geq 0$. 
Then both $A$ and $\J$ are maximally monotone, and so is their sum $A+\J$ 
thanks to a result by Heisler (see, e.g., \cite[Theorem~40.4]{Simons1}).
If $X$ is a Hilbert space, then $\J = \Id$ and it is well known
that the range $\ran(A+\lambda\J)$ is equal to $X$.
In striking contrast, it was shown very recently by Simons
that for the so-called Gossez operator $G$ which acts on $\ell_1$,  we have
that $\cran(G+\lambda\J)$ is \emph{not convex} for $0<\lambda<4$
(see Section~\ref{sec:Gossez} below.)
It is also known that for the so-called Fitzpatrick-Phelps operator $F$,
which acts on $L_1[0,1]$, the set $\cran(F+1\J)$ is \emph{not convex}.

In this paper, we unify these results by providing an abstract result
that allows us to deduce that no matter how $\lambda>0$ is chosen,
neither $\cran(F+\lambda\J)$ nor $\cran(G+\lambda\J)$ is convex.
This provides a negative answer to Simons's \cite[Problem~3.6]{Simonsnew}. 

Let us note here that neither $F$ nor $G$ is a subdifferential operator of a convex function;
indeed, if $f\colon X\to\RX$ is convex, lower semicontinuous, and proper, 
then $f+\lambda\tfrac{1}{2}\|\cdot\|^2$ is supercoercive\footnote{Recall that $g\colon X\to\RX$ is supercoercive if $\lim_{\|x\|\to+\infty}g(x)/\|x\|=+\infty$.} and hence 
$\cran\partial(f+\lambda\tfrac{1}{2}\|\cdot\|^2) = 
\cran(\partial f+\lambda\J ) = X^*$ by Gossez's \cite[Corollaire~8.2]{Gossez71}. 

The remainder of this note is organized as follows.
In Section~\ref{sec:aux} we collect some auxiliary results for 
later use. Section~\ref{sec:main} contains our main abstract results.
Section~\ref{sec:Gossez} discusses the Gossez operator while
Section~\ref{sec:FP} deals with the Fitzpatrick-Phelps operator. 

\section{Auxiliary results}
\label{sec:aux}

In this section, we cover some technical results that will ease 
the proofs in subsequent sections. 

\subsection*{Rugged Banach spaces}

\begin{definition}
(See also \cite{BauThesis}.)
We say that $X$ is \emph{rugged} if
\begin{equation}
\label{e:def:rugged}
\cspan\ran(\J-\J ) = X^*. 
\end{equation}
\end{definition}

\begin{example}
\label{ex:rugged}
(See \cite[Remark~5.6]{BauBor99}
and also
\cite[Section~2.3.4]{BauThesis}.)
The following Banach spaces are rugged:
\begin{enumerate*}
\item
\label{ex:rugged:i}
$\ell_1$ and 
\item
\label{ex:rugged:v}
$L_1[0,1]$.  
\end{enumerate*}
\end{example}

\begin{proof}
First, we define $(\forall \xi\in \RR)$ $\Sign  (\xi)=\{\sign (\xi)\}$, if $\xi\neq 0$; 
and $\Sign (\xi)=\left[-1,1\right]=\partial|\cdot|(0)$, otherwise.

\ref{ex:rugged:i}:
Let $x=(x_n)_{n\in \NN}\in \ell_1$, where 
$\NN = \{1,2,\ldots\}$. 
We start by proving the following claim:
\begin{equation}
\label{e:July:e6}
(\forall \nnn)
\quad
(\J x)_n
=\norm{x}_1\Sign (x_n).
\end{equation}  
Indeed, 
let $y=(y_n)_{\nnn}\in\ell_\infty$.
If 
$(\forall \nnn)$ 
$y_n\in \norm{x}_1\Sign(x_n)$, then 
$\norm{x}_1^2=\norm{y}_\infty^2=\scal{x}{y}$ and
hence $y\in \J x$.
Conversely, suppose that $y\in \J x$.
Then 
\begin{equation}
\label{e:July:e3}
(\forall\nnn)
\quad
\abs{y_n}\le \norm{y}_\infty=\norm{x}_1
\end{equation}
and 
\begin{subequations}
\label{e:July:e4}
\begin{align}
\norm{x}_1^2
&=\norm{y}^2_\infty
=\norm{x}_1\norm{y}_\infty
=\scal{x}{y}
=\sum_{\nnn}x_n y_n\\
&\le \sum_{\nnn}\abs{x_n} \abs{y_n}
\le \sum_{\nnn}\abs{x_n} \norm{y}_\infty
=\norm{x}_1\norm{y}_\infty.
\end{align}
\end{subequations}
Hence $\sum_{\nnn} \abs{x_n}(\norm{y}_\infty-\abs{y_n})=0 $; equivalently,
$(\forall \nnn)$
$ \abs{x_n}(\norm{y}_\infty-\abs{y_n})=0$. 
Now, if $x_n=0$ then \eqref{e:July:e3} implies that
$y_n\in\norm{x}_1\Sign (x_n)$.
Alternatively, if $x_n\neq 0$ we must have 
$\abs{y_n}=\norm{y}_\infty$
which, in view of  \eqref{e:July:e4}, 
implies 
that $y_n=
\sign (x_n)\abs{y_n}
=\sign (x_n)\norm{y}_\infty
\in\Sign (x_n) \norm{x}_1$. 
Next, we show that $\ell_1$ satisfies \eqref{e:def:rugged}.
Clearly, $\{1,2\}\subseteq\NN$,
and we denote the corresponding canonical unit vectors in $\ell_1$
by $e_1$ and $e_2$, respectively. 
If $i\in\{1,2\}$ and $n\in\NN\smallsetminus\{i\}$, then 
\eqref{e:July:e6} yields 
$(\J {e_i})_i=\{1\}$ and $(\J {e_i})_n = [-1,1]$; consequently, 
$(\J {e_i}-\J {e_i})_i=\{0\}$ and $(\J {e_i} - \J{e_i})_n = [-2,2]$. 
Therefore, if $i\in\{1,2\}$ and $n\in\NN\smallsetminus\{1,2\}$, then 
$(\J {e_1}-\J {e_1} + \J e_2 - \J e_2)_i=[-2,2]$ while 
$(\J {e_1}-\J {e_1} + \J e_2 - \J e_2)_n=[-4,4]$.
It follows that $(\forall \nnn)$
$[-2,2]\subseteq (\J {e_1}-\J {e_1} + \J e_2 - \J e_2)_n$, which in turn
implies \eqref{e:def:rugged}.

\ref{ex:rugged:v}:
The proof parallels that of item~\ref{ex:rugged:i};
for completeness, we provide the details. 
Let $x\in L_1[0,1]$. 
We first claim that 
\begin{equation}
\label{e:July:e6L}
(\forall t\in [0,1])
\quad
(\J x)(t)
=\norm{x}_1\Sign (x(t))\;\;\text{almost everywhere (a.e.). }
\end{equation}  
Indeed, 
let $y=L_\infty[0,1]$.
If 
$(\forall t\in[0,1])$ 
$y(t)\in \norm{x}_1\Sign(x(t))$ a.e., then 
$\norm{x}_1^2=\norm{y}_\infty^2=\scal{x}{y}$ and
hence $y\in \J x$.
Conversely, suppose that $y\in \J x$.
Then 
\begin{equation}
\label{e:July:e3L}
(\forall t\in [0,1])
\quad
\abs{y(t)}\le \norm{y}_\infty=\norm{x}_1 \;\;\text{a.e.}
\end{equation}
and 
\begin{subequations}
\label{e:July:e4L}
\begin{align}
\norm{x}_1^2
&=\norm{y}^2_\infty
=\norm{x}_1\norm{y}_\infty
=\scal{x}{y}
= \int_{0}^{1} x(t)y(t)\,dt\\
&\le 
\int_{0}^{1} \abs{x(t)}\abs{y(t)}\,dt 
\le \int_{0}^{1} \abs{x(t)} \norm{y}_\infty \, dt
=\norm{x}_1\norm{y}_\infty.
\end{align}
\end{subequations}
Hence 
$\int_{0}^{1}\abs{x(t)}(\norm{y}_\infty-\abs{y(t)})\,dt = 0$; 
equivalently, 
$(\forall t\in [0,1])$
$ \abs{x(t)}(\norm{y}_\infty-\abs{y(t)})=0$ a.e.
Now if $x(t)=0$, then \eqref{e:July:e3L} implies that
$y(t)\in\norm{x}_1\Sign (x(t))$.
Alternatively, if $x(t)\neq 0$, then we must have 
$\abs{y(t)}=\norm{y}_\infty$
which, in view of  \eqref{e:July:e4L}, 
implies 
that $y(t)=
\sign (x(t))\abs{y(t)}
=\sign (x(t))\norm{y}_\infty
\in\Sign (x(t)) \norm{x}_1$ a.e. 
It remains to show that $L_1[0,1]$ satisfies \eqref{e:def:rugged}.
Set $A_1 = [0,1/3]$ and $A_2 = [2/3,1]$
and also $e_1 = 3\chi_{A_1}$
and $e_2 = 3\chi_{A_2}$, where 
$\chi_B$ denotes the characteristic function of a subset $B$ of $[0,1]$.
If $i\in\{1,2\}$, $s\in A_i$ and $t\in[0,1]\smallsetminus A_i$, then
$\norm{e_i}_1=1$ and 
\eqref{e:July:e6L} yields 
$(\J {e_i})(s)=\{1\}$ and $(\J {e_i})(t) = [-1,1]$; 
consequently, 
$(\J {e_i}-\J {e_i})(s)=\{0\}$ and $(\J {e_i} - \J{e_i})(t) = [-2,2]$. 
Therefore, if $s\in A := A_1\cup A_2$ and $t \in[0,1]\smallsetminus A$, then 
$(\J {e_1}-\J {e_1} + \J e_2 - \J e_2)(s)=[-2,2]$ while 
$(\J {e_1}-\J {e_1} + \J e_2 - \J e_2)(t)=[-4,4]$.
It follows that $(\forall t\in [0,1])$
$[-2,2]\subseteq (\J {e_1}-\J {e_1} + \J e_2 - \J e_2)(t)$ a.e., 
which in turn implies \eqref{e:def:rugged}. 
\end{proof}

\begin{proposition}
\label{p:Shawn}
Let $X$ be rugged, 
let $A\colon X\to X^*$ be a linear operator,
and let $\lambda > 0$. 
Then 
$\cconv\cran(A+\lambda\J ) = X^*$.
\end{proposition}
\begin{proof}
We prove this by contradiction and thus assume
that there exists
$x^*\in X^*\smallsetminus (\cconv\cran(A+\lambda\J ))$. 
The separation theorem yields $x^{**} \in X^{**}\smallsetminus\{0\}$ 
such that  $\scal{x^*}{x^{**}} > \sup\scal{\cconv\cran(A+\lambda\J )}{x^{**}}
\geq \sup\scal{\ran(A+\lambda\J)}{x^{**}}$. 
Because $\ran(A+\lambda\J )$ is a balanced cone (i.e., closed under scalar multiplication),
 we deduce that 
$(\forall x\in X)$ 
$\scal{Ax+\lambda \J x}{x^{**}} = 0$. 
Because $A$ is single-valued and $\lambda\neq 0$, it follows that 
$(\forall x\in X)$ 
$\scal{\J x - \J x}{x^{**}} = 0$. 
Therefore, $\scal{X^*}{x^{**}} = \scal{\cspan\ran(\J -\J)}{x^{**}} = 0$
and thus $x^{**}=0$ which is absurd.
\end{proof}

\begin{corollary}
\label{c:thesis}
Let $X$ be rugged, let $A\colon X\to X^*$ be a linear operator,
and let $\lambda > 0$.
Then the following are equivalent:
\begin{enumerate}
\item 
\label{c:thesis1}
$\cran(A+\lambda\J) = X^*$;
\item 
\label{c:thesis2}
$\cran(A+\lambda\J)$ is a subspace;
\item 
\label{c:thesis3}
$\cran(A+\lambda\J)$ is a convex set. 
\end{enumerate}
\end{corollary}
\begin{proof}
(See also \cite[Proposition~15.3.8]{BauThesis} for ``\ref{c:thesis1}$\Leftrightarrow$\ref{c:thesis3}''.)
The implications ``\ref{c:thesis1}$\Rightarrow$\ref{c:thesis2}''
and ``\ref{c:thesis2}$\Rightarrow$\ref{c:thesis3}'' are clear.
``\ref{c:thesis3}$\Rightarrow$\ref{c:thesis1}'':
Assume that $\cran(A+\lambda\J)$ is convex.
Then $\cconv\cran(A+\lambda\J)=\cran(A+\lambda\J)$.
On the other hand, by Proposition~\ref{p:Shawn}, $\cconv\cran(A+\lambda\J)=X^*$.
Altogether, we deduce that $\cran(A+\lambda\J)=X^*$. 
\end{proof}

\section{Main result}
\label{sec:main}

\begin{lemma} \label{l:over}
Let $A\colon X\to X^*$ be a bounded linear monotone operator. 
Let $\lambda > 0$ and let $r^*\in X^*$.
Suppose that $x\in X$ satisfies 
\begin{equation}
r^* \in Ax+\lambda \J x. 
\end{equation}
Then 
\begin{equation}
\label{e:over2}
\frac{\|r^*\|}{\|A\|+\lambda} \leq \|x\| \leq \frac{\|r^*\|}{\lambda}
\quad \text{and}\quad
\lambda\|x\|^2 \leq \scal{x}{r^*}, 
\end{equation}
and\footnote{Recall that $A\colon X\to X^*$ is skew
if any of the following equivalent conditions holds:
(i) $A^*|_X = -A$;
(ii)
$(\forall x\in X)(\forall y\in X)$ $\scal{x}{Ay} = -\scal{y}{Ax}$;
(iii) $\pm A$ are monotone;
(iv) $(\forall x\in X)$ $\scal{x}{Ax} = 0$.
}

\begin{equation}
\label{e:over3}
\text{$A$ is skew}
\quad\Rightarrow\quad
\lambda\|x\|^2 = \scal{x}{r^*}.
\end{equation}
\end{lemma}
\begin{proof}
Let $x^*\in \J x$ satisfy $r^* = Ax + \lambda x^*$. 
Combining \eqref{e:J} and the triangle inequality yields
$\lambda\|x\| = \lambda\|x^*\|=\|r^*-Ax\| \geq \|r^*\|-\|Ax\| \geq \|r^*\|-\|A\|\|x\|$. 
Hence $(\lambda+\|A\|)\|x\|\geq\|r^*\|$ which implies the first inequality in \eqref{e:over2}. 
On the other hand, from \eqref{e:J} and the monotonicity of $A$, we deduce that 
$\lambda \|x\|^2 = \lambda\scal{x}{x^*} \leq \scal{x}{\lambda x^* + Ax} 
= \scal{x}{r^*}\leq \|x\|\|r^*\|$, which yields the remaining inequalities in \eqref{e:over2}
as well as \eqref{e:over3}.
\end{proof}

\begin{theorem}
\label{t:tuesday}
Let $A\colon X\to X^*$ be a bounded linear monotone operator. 
Let $\lambda > 0$ and $\varepsilon\geq 0$. 
Suppose that $x\in X$ and $f^*\in X^*$ satisfy\footnote{We denote 
the closed ball centered at $f^*$ of radius $\varepsilon$ by 
$\BB{f^*}{\varepsilon}$.}
\begin{equation}
\big(Ax + \lambda\J x\big) \cap \BB{f^*}{\varepsilon} \neq \varnothing
\label{e:tuesday2}
\end{equation}
and 
\begin{equation}
\varepsilon \leq \frac{2\lambda\|f^*\|}{\|A\|+3\lambda}.
\label{e:tuesday3}
\end{equation}
Then 
\begin{equation}
\scal{x}{f^*} \geq \lambda\left(\frac{\|f^*\|-\varepsilon}{\|A\|+\lambda}\right)^2
-\varepsilon\left(\frac{\|f^*\|-\varepsilon}{\|A\|+\lambda} \right). 
\label{e:tuesday4}
\end{equation}
\end{theorem}
\begin{proof}
Set $\varphi = \|f^*\|$ and let 
$r^* \in (Ax + \lambda\J x) \cap \BB{f^*}{\varepsilon}$. 
Combining this with \eqref{e:over2}, we obtain
\begin{equation}
\lambda\|x\|^2 \leq \scal{x}{r^*} = \scal{x}{r^*-f^*}+\scal{x}{f^*}
\leq \|x\|\varepsilon + \scal{x}{f^*}.
\label{e:tuesday6}
\end{equation}
Next, set $\rho = \|r^*\|$, $\alpha = \|A\|$, and $\xi = \|x\|$.
Then \eqref{e:tuesday6} and \eqref{e:over2} yield
\begin{equation}
\scal{x}{f^*} \geq \lambda\xi^2 - \varepsilon\xi,
\quad\text{where}\quad
\frac{\rho}{\alpha+\lambda} \leq \xi \leq \frac{\rho}{\lambda}. 
\label{e:tuesday8}
\end{equation}
Note that 
$\|f^*\|-\|f^*-r^*\| \leq \|r^*\| \leq \|f^*\| + \|r^*-f^*\|$ which implies 
\begin{equation}
\varphi - \varepsilon \leq \rho \leq \varphi + \varepsilon.
\label{e:tuesday10}
\end{equation}
Now, set $l=(\varphi-\varepsilon)/(\alpha+\lambda)\geq 0$.
On the one hand, 
simple algebraic manipulations show that \eqref{e:tuesday3} is equivalent to
$\varphi/(\alpha+3\lambda)\le l$; 
consequently, $2\lambda\varphi/(\alpha+3\lambda)\le 2\lambda l$.
On the other hand, \eqref{e:tuesday3} states that 
$\varepsilon \leq 2\lambda\varphi/(\alpha+3\lambda)$. Altogether, 
\begin{equation}
\label{e:July:1}
\varepsilon\le 2\lambda l.
\end{equation}
Furthermore, \eqref{e:tuesday8} and \eqref{e:tuesday10} yield
$\xi \geq \rho/(\alpha+\lambda)\geq (\varphi-\varepsilon)/(\alpha+\lambda) = l$;
hence, 
\begin{equation}
\label{e:July:2}
l\le \xi.
\end{equation}
Next, consider the function $g(\eta)=\lambda\eta^2 - \varepsilon\eta$.
Let $\eta \ge l $ 
and observe that \eqref{e:July:1} implies that
$g'(\eta)=2\lambda\eta - \varepsilon\ge 2\lambda l- \varepsilon\ge 0 $. 
Hence, $g$
is increasing on $\left[l,+\infty\right [$.
Combining \eqref{e:tuesday8} and 
\eqref{e:July:2}, we learn that $\scal{x}{f^*}\ge g(\xi)\ge g(l)$.
Hence $\scal{x}{f^*} \geq g(l)$, which is precisely \eqref{e:tuesday4}. 
\end{proof}

\begin{definition}[\textbf{\whs condition}]
Let $A\colon X\to X^*$ be a bounded linear monotone operator,
and let $f^*\in X^*$.
We say that the \emph{\whs condition}\footnote{``whs'' stands for \emph{wondrous half-space}.} holds if
$\|A\|=1$, $\|f^*\|=1$, and 
for every $\lambda>0$ and every $\varepsilon >0$, the implication
\begin{equation}
(Ax + \lambda \J x)\cap\BB{f^*}{\varepsilon} \neq\varnothing
\quad\Rightarrow\quad
\scal{x}{f^*} \leq 3\varepsilon
\end{equation}
is true. 
\end{definition}

For every $\lambda>0$, we define for future convenience
\begin{empheq}[box=\mybluebox]{equation}
\label{e:defofm}
m(\lambda) = \frac{3\lambda^2+9\lambda+4-(\lambda+1)\sqrt{9\lambda^2+36\lambda+16}}{4\lambda+2}. 
\end{empheq}
The function $m$ plays a role in the statement and proofs of this section. 
The next result will also establish the strict positivity\footnote{In passing, 
we note that 
$\lim_{\lambda\to 0^+} m(\lambda) = 0$, 
$\lim_{\lambda\to +\infty} m(\lambda) = 0$,
$m$ is strictly increasing on $[0,\bar{\lambda}]$
and strictly decreasing on $\left[\bar{\lambda},+\infty\right[$,
where $\bar{\lambda} = (3+\sqrt{15})/6 \approx 1.14550$,
and $m(\bar{\lambda}) = (9-2\sqrt{15})/(12+2\sqrt{15})\approx 0.06349$.
We omit the proofs as these properties are not needed in this paper.} 
of $m$.

\begin{corollary}
\label{c:tuesday}
Let $A\colon X\to X^*$ be a bounded linear monotone operator with $\|A\|=1$, 
let $f^*\in X^*$ be such that $\|f^*\|=1$,
let $\lambda > 0$ and assume that there exist $x\in X$ and 
$\varepsilon > 0$ such that $\varepsilon \leq 2\lambda/(1+3\lambda)$ and 
$(Ax + \lambda \J x)\cap\BB{f^*}{\varepsilon} \neq\varnothing$.
Finally, assume that the \whs condition holds.
Then 
\begin{equation}
(2\lambda+1)\varepsilon^2 -(3\lambda^2+9\lambda+4)\varepsilon +\lambda\leq 0;
\label{e:tuesday32}
\end{equation}
consequently,
\begin{equation}
\varepsilon \geq \frac{3\lambda^2+9\lambda+4-(\lambda+1)\sqrt{9\lambda^2+36\lambda+16}}{4\lambda+2} = m(\lambda)> 0.
\label{e:tuesday34}
\end{equation}
\end{corollary}
\begin{proof}
The \whs condition implies that $\scal{x}{f^*} \leq 3\varepsilon$. 
On the other hand, \eqref{e:tuesday4} yields
$\scal{x}{f^*} \geq \lambda (1-\varepsilon)^2/(1+\lambda)^2 - \varepsilon(1-\varepsilon)/(1+\lambda)$. 
Altogether,
\begin{equation}
\lambda \frac{(1-\varepsilon)^2}{(1+\lambda)^2} - \varepsilon\frac{1-\varepsilon}{1+\lambda} \leq 3\varepsilon. 
\label{e:tuesday30}
\end{equation}
In turn, \eqref{e:tuesday30} is equivalent to
\begin{subequations}
\begin{align}
&\qquad\lambda(1-\varepsilon)^2 - \varepsilon(1-\varepsilon)(1+\lambda) \leq 3\varepsilon(1+\lambda)^2
\\
&\Leftrightarrow
(\lambda + \varepsilon^2\lambda - 2\lambda\varepsilon )
+(-\varepsilon -\varepsilon\lambda+\varepsilon^2+\varepsilon^2\lambda)
\leq 3\varepsilon + 3\varepsilon\lambda^2+6\varepsilon\lambda
\\
&\Leftrightarrow
\varepsilon^2(1+2\lambda) +\varepsilon(-9\lambda-4-3\lambda^2) + \lambda \leq 0,
\end{align}
\end{subequations}
which is equivalent to \eqref{e:tuesday32}.
Now let's view \eqref{e:tuesday32} as a quadratic inequality.
Then $\varepsilon$ must lie in the closed interval given by the roots of the 
corresponding quadratic equation
\begin{equation}
(2\lambda+1)\xi^2 -(3\lambda^2+9\lambda+4)\xi +\lambda = 0,
\label{e:tuesday36}
\end{equation}
where the variable is $\xi$. The two roots of \eqref{e:tuesday36} are
\begin{subequations}
\begin{align}
&\;\quad \frac{3\lambda^2+9\lambda+4\pm \sqrt{(3\lambda^2+9\lambda+4)^2-4(1+2\lambda)\lambda}}{2(1+2\lambda)}\\
&= \frac{3\lambda^2+9\lambda+4\pm (\lambda+1)\sqrt{9\lambda^2+36\lambda+16}}{4\lambda+2}, 
\end{align}
\end{subequations}
and both roots are positive.
\end{proof}

\begin{theorem}[\textbf{key result}]
\label{t:tuesdaypm}
Let $A\colon X\to X^*$ be a bounded linear monotone operator with $\|A\|=1$, 
let $f^*\in X^*$ be such that $\|f^*\|=1$,
and let $\lambda > 0$. 
Furthermore, assume the \whs condition holds. 
Then either 
\begin{equation}
d\big(f^*,\ran(A+\lambda\J)\big)\geq \frac{2\lambda}{1+3\lambda},
\end{equation}
or 
\begin{equation}
\frac{2\lambda}{1+3\lambda} > d\big(f^*,\ran(A+\lambda\J)\big)
\geq 
\frac{3\lambda^2+9\lambda+4-(\lambda+1)\sqrt{9\lambda^2+36\lambda+16}}{4\lambda+2} > 0.
\end{equation}
In any case,
\begin{equation}
d\big(f^*,\ran(A+\lambda\J)\big)\geq 
\frac{3\lambda^2+9\lambda+4-(\lambda+1)\sqrt{9\lambda^2+36\lambda+16}}{4\lambda+2} > 0.
\end{equation}
\end{theorem}
\begin{proof}
Assume that the ``either'' alternative fails, i.e., 
$d(f^*,\ran(A+\lambda\J))< 2\lambda/(1+3\lambda)$. 
Then there exist $x_n\in X$ and $x_n^*\in \J x_n$ such that 
$\varepsilon_n := \|f^*-(Ax_n+\lambda x_n^*)\| \leq 2\lambda/(1+3\lambda)$ and
$\varepsilon_n \downarrow d(f^*,\ran(A+\lambda\J))$. 
Now \eqref{e:tuesday34} of Corollary~\ref{c:tuesday} yields 
\begin{equation}
\varepsilon_n\geq 
\frac{3\lambda^2+9\lambda+4-(\lambda+1)\sqrt{9\lambda^2+36\lambda+16}}{4\lambda+2} > 0.
\end{equation}
Hence the ``or'' case follows by letting $n\to+\infty$. 
Finally, we claim that 
\begin{equation}
\frac{2\lambda}{1+3\lambda} > \frac{3\lambda^2+9\lambda+4-(\lambda+1)\sqrt{9\lambda^2+36\lambda+16}}{4\lambda+2}.
\end{equation}
It is straightforward but a bit tedious to verify that 
\begin{equation}
\label{e:wolfram}
\frac{2\lambda}{1+3\lambda} - \frac{3\lambda^2+9\lambda+4-(\lambda+1)\sqrt{9\lambda^2+36\lambda+16}}{4\lambda+2}
= \frac{1+\lambda}{2(1+2\lambda)(1+3\lambda)}\cdot\tau,
\end{equation}
where
\begin{equation}
\tau = (3\lambda+1)\sqrt{9\lambda^2+36\lambda+16} - \big(9\lambda^2+13\lambda+4\big).
\end{equation}
It remains to show that  $\tau> 0$. 
Indeed,
\begin{subequations}
\begin{align}
\tau> 0
&\Leftrightarrow
(3\lambda+1)^2(9\lambda^2+36\lambda+16) > (9\lambda^2+13\lambda+4)^2
\\
&\Leftrightarrow
(3\lambda+1)^2(9\lambda^2+36\lambda+16) - (9\lambda^2+13\lambda+4)^2 
= 144\lambda^3+128\lambda^2+28\lambda > 0.
\end{align}
\end{subequations}
\end{proof}

We are now ready for our abstract main result.

\begin{corollary}[\textbf{main result}]
\label{c:main}
Suppose that $X$ is rugged, 
let $A\colon X\to X^*$ be a bounded linear monotone operator
with $\|A\|=1$, let $f^*\in X^*$ be such that $\|f^*\|=1$,
and assume that the \whs condition holds. 
Let $\lambda > 0$, 
set $R = \cran(A+\lambda\J )$, and let $\alpha\in\RR\smallsetminus\{0\}$.
Then $R$ is \emph{not convex}, $\cconv R = X^*$, and 
\begin{equation}
d(\alpha f^*,R) \geq |\alpha|m(\lambda)>0,
\end{equation}
where 
\begin{equation}
m(\lambda) = \frac{3\lambda^2+9\lambda+4-(\lambda+1)\sqrt{9\lambda^2+36\lambda+16}}{4\lambda+2} > 0. 
\end{equation}
\end{corollary}
\begin{proof}
Because $A$ and $\J$ are homogeneous,
it follows that $\alpha R = R$.
Hence
$d(\alpha f^*,R) = d(\alpha f^*,\alpha R) 
=|\alpha|d(f^*,R) \geq |\alpha|m(\lambda)$
by Theorem~\ref{t:tuesdaypm}.
The lack of convexity of $R$ and the fact that $\cconv R = X^*$ follow 
from Corollary~\ref{c:thesis} and Proposition~\ref{p:Shawn}, respectively.
\end{proof}

\section{The Gossez operator revisited}
\label{sec:Gossez}

In this section, we assume that 
\begin{empheq}[box=\mybluebox]{equation}
X = \ell_1,
\end{empheq}
and that 
\begin{empheq}[box=\mybluebox]{equation}
G \colon \ell_1\to\ell_\infty \colon x=(x_n)_{\nnn} \mapsto (Gx)_\nnn, 
\quad \text{where $(Gx)_n = -\sum_{k<n}x_k + \sum_{k>n}x_k$}
\end{empheq}
is the Gossez operator \cite{Gossez72} and \cite{Gossez76}.
It is easy to see that $G$ is a bounded linear operator 
with $\|G\|=1$ and that $G$ is skew, hence monotone. 

We are now ready for the main result concerning
the Gossez operator.

\begin{theorem}
\label{t:Gossez}
The Gossez operator $G$ satisfies the \whs condition with
$f^*=-(1,1,\ldots)\in\ell_\infty$; consequently, 
$(\forall \lambda>0)$
$\cran(G+\lambda\J )$ is not convex yet $\cconv\cran(G+\lambda\J)=\ell_\infty$. 
\end{theorem}
\begin{proof}
Let $\lambda>0$, and 
let $\varepsilon > 0$.
Assume that $(x,r^*)\in \ell_1\times\ell_\infty$ satisfies
$r^* \in (Gx+\lambda\J x)\cap \BB{f^*}{\varepsilon}$.
First, there exists $x^*\in\J x$ such that $r^*=Gx+\lambda x^*$.
By definition of $G$, we have $(\forall\nnn)$
$r_n^* = -\sum_{k<n}x_k + \sum_{k>n}x_k+\lambda x_n^*$. 
Letting $n\to\infty$, we deduce that 
\begin{equation}
\label{e:0903a}
\lim_{n\to\infty} (\lambda x_n^*-r_n^*) = \sum_{k=1}^\infty x_k = \scal{x}{-f^*} =: \sigma. 
\end{equation}
Second, because $r^*\in\BB{f^*}{\varepsilon}$, 
we have $\|r^*-f^*\|\leq\varepsilon$. Thus 
$(\forall\nnn)$ $|r_n^*+1|\leq\varepsilon$ 
and so $r_n^* \leq -1+\varepsilon$. 
Altogether, for all $n$ \emph{sufficiently large} we have 
$\lambda x_n^*-r_n^* \leq \sigma + \varepsilon$ 
\begin{equation}
\lambda x_n^* =(\lambda x_n^*-r_n^*) + r_n^* 
\leq (\sigma+\varepsilon)+(-1+\varepsilon) = \sigma+2\varepsilon-1. 
\label{e:wednesday10}
\end{equation}
Using \eqref{e:J} and \eqref{e:over2}, we estimate
\begin{equation}
-\lambda x_n^* \leq \lambda\|x^*\| = \lambda \|x\| \leq \|r^*\|
\leq \|r^*-f^*\| + \|f^*\| \leq \varepsilon + 1. 
\label{e:wednesday12}
\end{equation}
Adding 
\eqref{e:wednesday10} and \eqref{e:wednesday12},
we obtain 
$0\leq \sigma+3\varepsilon$.
Recalling the definition of $\sigma$ from \eqref{e:0903a}, 
deduce that 
$\scal{x}{f^*}\leq 3\varepsilon$ and the \whs condition holds.
The conclusion now follows by applying 
Example~\ref{ex:rugged}\ref{ex:rugged:i} and 
Corollary~\ref{c:main}.
\end{proof}

\begin{remark}[\textbf{a negative answer to a question posed by Stephen Simons}]
In \cite{Simonsnew}, Stephen Simons
proved that $\cran(G+\lambda\J )$ is not convex
for $0<\lambda<4$. In \cite[Problem~3.6]{Simonsnew}
he asks whether $\cran(G+4\J )$ is convex.
Theorem~\ref{t:Gossez} not only provides a negative
answer but also establishes,
for \emph{every} $\lambda>0$, 
the nonconvexity of $\cran(G+\lambda\J )$. 
\end{remark}

\begin{remark}[\textbf{the negative Gossez operator}]
The \emph{negative Gossez operator}, $-G$, is much better behaved than $G$:
indeed, combining \cite[Example~14.2.2 and Theorem~15.3.7]{BauThesis}, 
we deduce that $(\forall \lambda>0)$ $\cran(-G+\lambda\J) = \ell_\infty$.
(See also \cite[Example~5.2]{BauBor99}.)
\end{remark}

\section{The Fitzpatrick-Phelps operator revisited}
\label{sec:FP}

In this section, we assume that 
\begin{empheq}[box=\mybluebox]{equation}
X = L_1[0,1]
\end{empheq}
and that 
\begin{empheq}[box=\mybluebox]{equation}
F \colon L_1[0,1]\to L_\infty[0,1] \colon x \mapsto Fx,
\quad \text{where $(Fx)(t) = \int_{0}^t x(s)\,ds -\int_{t}^1x(s)\,ds$.}
\end{empheq}
It is easy to see that $F$ is a bounded linear operator with $\|F\|=1$
and that $F$ is skew, hence monotone. 

We are now ready for the main result concerning the Fitzpatrick-Phelps operator.

\begin{theorem}
\label{t:FP}
The Fitzpatrick-Phelps operator $F$ satisfies the \whs condition with
$f^*\equiv -1 \in L_\infty[0,1]$; 
consequently, $(\forall \lambda>0)$ 
$\cran(F+\lambda\J )$ is not convex yet $\cconv\cran(F+\lambda\J) = L_\infty[0,1]$. 
\end{theorem}
\begin{proof}
Let $\lambda>0$, and let $\varepsilon > 0$.
Assume that $(x,r^*)\in L_1[0,1]\times L_\infty[0,1]$ satisfies 
\begin{equation}
r^* \in (Fx+\lambda\J x)\cap \BB{f^*}{\varepsilon}.
\end{equation}
First, there exists $x^*\in\J x$ such that $r^*=Fx+\lambda x^*$.
By definition of $F$, $(\forall t\in[0,1])$ 
$r^*(t) = \int_{0}^{t}x(s)\,ds - \int_{t}^{1}x(s)\,ds + \lambda x^*(t)$.
Thus, for every $t\in[0,1]$, we have $|r^*(t)+1|\leq\varepsilon$ and also 
\begin{subequations}
\begin{align}
\lambda x^*(t)
&= r^*(t) - \int_{0}^{t} x(s)\,ds + \int_{t}^{1}x(s)\,ds 
= r^*(t)+\int_{0}^{1}x(s)\,ds - 2\int_{0}^{t}x(s)\,ds\\
&= r^*(t) - \scal{x}{f^*} - 2\int_{0}^{t}x(s)\,ds
\leq -1+\varepsilon -\scal{x}{f^*} - 2\int_{0}^{t}x(s)\,ds.
\end{align}
\end{subequations}
In view of the (absolute) continuity of $t\mapsto \int_{0}^{t} x(s) \,ds$ 
(see, e.g., \cite[Theorem~6.84]{Stromberg}), we have 
for all $t>0$ \emph{sufficiently small}, 
\begin{equation}
\lambda x^*(t) \leq - 1 - \scal{x}{f^*} + 2\varepsilon. 
\label{e:thursday10}
\end{equation}
Now note that --- 
using \eqref{e:J} and again \eqref{e:over2} --- we obtain 
\begin{equation}
-\lambda x^*(t) \leq 
\lambda\|x^*\|=\lambda\|x\|
\le \|r^*\|
\le \|r^*-f^*\| + \|f^*\|\leq \varepsilon + 1.
\label{e:thursday12L}
\end{equation}
Adding \eqref{e:thursday10} and \eqref{e:thursday12L}, 
we deduce that 
$0 \leq 3\varepsilon - \scal{x}{f^*}$
and the \whs condition  thus holds.
Finally, apply Example~\ref{ex:rugged}\ref{ex:rugged:v} 
and Corollary~\ref{c:main}.
\end{proof}

\begin{remark}
Fitzpatrick and Phelps (see \cite[Example~3.2]{FP}) showed directly that 
$\cran(F+1\J )$ is not convex. 
We extend their conclusion from $\lambda=1$ to any $\lambda>0$. 
We note that 
a referee pointed out to us that there is a gap
in the proof of the nonconvexity of $\cran(F+\J)$ in 
the second paragraph of \cite[page~64]{FP};
however, the work in this paper does not rely on their argument
and thus provides an alternative (and simpler) proof. 
\end{remark}

\begin{remark}[\textbf{the negative Fitzpatrick-Phelps operator}]
The \emph{negative Fitzpatrick-Phelps operator}, $-F$,
satisfies the \whs condition again with $f^*\equiv -1\in L_\infty[0,1]$.
(The proof is similar with the only difference being that in the derivation
of the counterpart of \eqref{e:thursday10}, we work with
$t$ sufficiently close to $1$ rather than to $0$.)
Consequently, $(\forall \lambda>0)$ $\cran(-F+\lambda\J)$ is not convex
yet $\cconv\cran(-F+\lambda\J)=L_\infty[0,1]$. 
\end{remark}

\section*{Acknowledgements}
We thank a referee for her/his pertinent comments and suggestions which resulted in 
significant simplifications in some of the proofs.
The research of HHB and XW was partially supported by Discovery Grants
of the Natural Sciences and Engineering Research Council of
Canada. 
The research of WMM was partially supported by 
the Natural Sciences and Engineering Research Council of
Canada Postdoctoral Fellowship.

% \small


\begin{thebibliography}{999}
\seppthree

\bibitem{BauThesis}
H.H.\ Bauschke, 
\emph{Projection Algorithms and Monotone Operators},
PhD thesis, Simon Fraser University, 1996.
\url{http://summit.sfu.ca/item/7015}

\bibitem{BauBor99}
H.H.\ Bauschke and J.M.\ Borwein,
Maximal monotonicity of dense type,
local maximal monotonicity, and 
monotonicity of the conjugate are 
all the same for continuous linear operators,
\emph{Pacific Journal of Mathematics}~189 (1999), 1--20. 

\bibitem{FP}
S.P.\ Fitzpatrick and R.R.\ Phelps,
Some properties of maximal monotone operators 
on nonreflexive Banach spaces,
\emph{Set-Valued Analysis}~3 (1995), 51--69. 

\bibitem{Gossez71}
J.-P.\ Gossez,
Op\'erateurs monotones non lin\'eaires dans
les espaces de Banach non r\'eflexifs,
\emph{Journal of Mathematical Analysis and Applications}~34 (1971), 371--395. 

\bibitem{Gossez72}
J.-P.\ Gossez,
On the range of a coercive maximal monotone operator in
a nonreflexive Banach space,
\emph{Proceedings of the American Mathematical Society}~35 (1972), 88--92. 

\bibitem{Gossez76}
J.-P.\ Gossez,
On a convexity property of the range of a maximal monotone
operator, 
\emph{Proceedings of the American Mathematical Society}~55 (1976), 359--360. 

\bibitem{Simonsnew}
S.\ Simons,
Gossez's skew linear map and its pathological maximally 
monotone multifunctions,
\emph{Proceedings of the American Mathematical Society} in press, 
\url{https://arxiv.org/abs/1807.06152}, 
DOI: \url{https://doi.org/10.1090/proc/14547}.

\bibitem{Simons1}
S.\ Simons,
\emph{Minimax and Monotonicity},
Springer-Verlag,
1998. 

\bibitem{Stromberg}
K.R.\ Stromberg,
\emph{Introduction to Classical Real Analysis},
Wadsworth International Group, 
1981.


\end{thebibliography}
\end{document}